\documentclass[reqno,12pt,letterpaper]{amsart}
\usepackage[proof]{sdmacros}
\usepackage{enumerate}
\usepackage{subfigure}
\usepackage{mathtools}
\mathtoolsset{showonlyrefs}

\newcommand{\abs}[1]{\lvert#1\rvert}

\newcommand{\RR}{{\mathbb R}}

\newcommand{\CC}{{\mathbb C}}

\title[control and damping]{Control estimates for 0th order pseudodifferential operators}
\author{Hans Christianson}
\email{hans@math.unc.edu}
\author{Jian Wang}
\email{wangjian@email.unc.edu}
\address{Department of Mathematics, University of North Carolina, Chapel Hill, NC 27514}
\author{Ruoyu P. T. Wang}
\email{rptwang@math.northwestern.edu}
\address{Department of Mathematics, Northwestern University, Evanston, IL 60208}

\begin{document}

\begin{abstract}
    We introduce the control conditions for 0th order pseudodifferential operators $\mathbf{P}$ whose real parts satisfy the Morse--Smale dynamical condition. We obtain microlocal control estimates under the control conditions. As a result, we show that there are no singular profiles in the solution to the evolution equation $(i\partial_t-\mathbf{P})u=f$ when $\mathbf{P}$ has a damping term that satisfies the control condition and $f\in C^{\infty}$. This is motivated by the study of a microlocal model for the damped internal waves.
\end{abstract}

\maketitle

\section{Introduction}

Zeroth order pseudodifferential operators are used as mathematical models for internal waves since the work of Colin de Verdi\`ere--Saint-Raymond \cite{attractor}. When the operator is \emph{self-adjoint} and the Hamiltonian flow of its principal symbol satisfies the Morse--Smale condition, \cite{attractor} proved that the solution to the evolution equation has singular concentration on the attractive Lagrangians for the Hamiltonian flow. Such concentration provides a microlocal interpretation of the formation of attractors in internal waves -- see \cite{maas}, \cite{brouzet} for the physics literature. Using tools from scattering theory, in particular the radial estimates introduced by Melrose \cite{mel95}, Dyatlov--Zworski \cite{force} provided an alternative proof for the singular concentration of the solutions with relaxed conditions on the operators. Much work has been done on 0th order operators since then: Colin de Verdi\`ere \cite{co2} studied 0th order operators in high dimensions and with weaker dynamical assumptions; Wang \cite{scattering} studied the scattering matrix for 0th order operators, which is analogous to the scattering matrix for the Schr\"odinger operators in the Euclidean spaces;  Galkowski--Zworski \cite{vis} studied the resonances for 0th order operators via viscosity limits; Wang \cite{fermi} studied the dynamics of resonances for 0th order operators, in particular, a Fermi golden rule for embedded eigenvalues is proved. Tao \cite{circle} studied the spectral theory for 0th order operators on the circle and found embedded eigenvalues for such operators. Almonacid--Nigam \cite{numerics} obtained numerical results for both the evolution equations and viscosity limits of eigenvalues. The recent work of Dyatlov--Wang--Zworski \cite{aquarium} proved the formation of internal wave attractors in two-dimensional domains. We also mention that there are other motivation to study 0th order operators: for example, Ralston \cite{ralston} considered the Poincar\'e problem and investigated the spectrum and the eigenvalues of a 0th order operator, in the study of rotating fluids, 

Motivated by studying the damping of internal wave attractors \cite{fb, bbsm, bl}, we consider a microlocal model for the damped internal waves in this paper. More precisely, for a closed smooth surface $M$, we consider \emph{non-self-adjoint} 0th order pseudodifferential operators $\mathbf P = P-i\chi\in \Psi^0(M)$, where $\Psi^0(M)$ is the set of 0th order pseudodifferential operators whose symbols have asymptotic expansion with homogeneous symbols, $P\in\Psi^0(M)$ is self-adjoint and $\chi\in C^{\infty}(M)$, $\chi\geq 0$. We assume $P$ satisfies the Morse--Smale condition as in \cite{attractor, force}, hence it models the internal waves without damping. The smooth function $\chi$, which is also in $\Psi^0(M)$ as a multiplication operator, is now considered as the ``damping'' term. We are interested in the evolution problem for $\mathbf P$. In particular, we show that if $\chi$ satisfies certain ``control condition'', then the internal wave attractors disappear and the solution stays bounded in $L^2$: see Theorem \ref{thm: damp}. 

As indicated in the previous paragraph, the main tools we use to study the damping problem are resolvent estimates, in the form of microlocal control estimates for 0th order operators: see Theorem \ref{thm: control}. Since the singularity of solutions to $(\mathbf P-\omega)u=0$ propagates along the bicharacteristics of $\Re \sigma(\mathbf P)$ (where the direction of propagation is determined by the sign of $\Im\sigma(\mathbf P)$), one can control the solution from a subset $K$ of the cotangent bundle $T^*M\setminus 0$, assuming $T^*M\setminus 0$ can be covered by finitely many images of the subset $K$ propagated in finite time. When the Hamiltonian flow of $\Re\sigma(\mathbf P)$ satisfies the Morse--Smale condition, and we assume the subset $K$ is small and located on the attractive (or repulsive) limit cycles, the repulsive (or attractive, respectively) limit cycles cannot be covered in the above-mentioned way. In these cases, we will use radial estimates from scattering theory \cite{mel95} -- see \cite{hmv, kds, fw, quasi, hv, zeta} for some of the recent developments. We also mention that since the Morse--Smale condition is only stable under small perturbations, we can only obtain the microlocal control for a small range of the spectral parameters. We refer to \cite{lions,miller,bz04} for introductions to control theory and resolvent estimates. 

Resolvent estimates often imply stability results for damped equations, especially the damped wave equation. Classical results \cite{rt74,blr92,bg97} in the damped wave equation, together with recent developments \cite{bj16,wan20,bg20,kk22,kle22b} in the settings of non-compact manifolds, singular, anisotropic, and time-dependent damping, implies that there is exponential energy decay without loss of derivatives if and only if the \emph{geometric control condition} is satisfied: the geometric control condition is a strong dynamical assumption that every geodesic enters the damped region in finite time. Moreover, the solution will be sufficiently smooth if both the damping and the initial data are sufficiently smooth. Even if there are, for example, periodic geodesics or other attracting or repulsing invariant sets which are outside the control regions, solutions are still smooth near the attractors, as in the vast control theory and damped wave literature \cite{CdVP-I,Chr-1,Chr-NC-erratum,aln14,CSVW14,bc15,ll17,wun17,sta17,kle19b,dk20,wan21,wan21b,kle22,sun22,kw22}, inexhaustively listed here.

In stark contrast, in this paper we study solutions to evolution equations of zeroth order, for which solutions near attractors exhibit some blow-up.  The goal of this work is to show if damping is introduced near the attractors, it stabilizes the solutions. Note that in this case, the control conditions we impose are weaker than the geometric control condition.

This problem is new and requires new techniques.  For the usual damped wave problem, the coercive nature of $-\Delta$ gives rise to a compact resolvent. Using the analytic Fredholm theory, this implies the damped wave resolvent has a meromorphic extension to a neighbourhood of the real axis.  In particular, the poles of the damped wave resolvent are of finite order and discrete, hence a contour deformation in the inverse Fourier transform allows one to conclude the exponential decay of energy.

On the other hand, for the zeroth order pseudodifferential operators studied in this paper, the resolvent is not compact in the $L^2$ space, so the analytic Fredholm theory does not apply.  In particular, a limiting absorption principle is non-trivial.  In fact, as one can see in the second remark after Theorem \ref{thm: damp}, the unique continuation principle can fail for zeroth order operators even with a spatial damping term.

\subsection{Assumptions}
\label{subsec: assumptions}
Let $M$ be a compact surface without boundary. Let $\mathbf P = P+iQ\in \Psi^0(M)$, where $P$, $Q$ are self-adjoint operators with respect to some density $dm$ on $M$: 
\begin{equation}
    P, Q\in \Psi^0(M), \ P^*=P, \ Q^*=Q \ \text{on} \ L^2(M, dm).
\end{equation}
Here $\Psi^0(M)$ is the space of 0th order polyhomogeneous pseudodifferential operators defined in \cite[Definition E.12]{dz19} with $h=1$. The operators $P$, $Q$ are called the real and the imaginary parts of $\mathbf P$ respectively, and we write
\[ \Re \mathbf P:= \frac{\mathbf P+\mathbf P^*}{2}=P, \ \Im \mathbf P:=\frac{\mathbf P- \mathbf P^*}{2i}=Q. \]
Let $\beta$ be defined as 
\[ \beta: \RR_+\times (T^*M\setminus 0)\to T^*M\setminus 0, \ \beta(r, x,\xi):=(x,r\xi). \]
$\beta$ gives an action of $\RR_+$ on $T^*M\setminus 0$. We denote the quotient space and the quotient map by 
\[ \partial \overline{T^*M}:= (T^*M\setminus 0)/\RR_+, \ \kappa: T^*M\setminus 0 \to \partial\overline{T^*M}. \]
$\partial\overline{T^*M}$ is a 3-dimensional orientable smooth manifold.

Let $p$, $q\in S^0(T^*M\setminus 0; \RR)$ be the principal symbols of $P$, $Q$. Then $p, q$ are homogeneous of order $0$. We assume that 
\begin{equation}
    \label{ap: rv}
    0 \text{ is a regular value of } p.
\end{equation}
In other words, $dp|_{p^{-1}(0)}\neq 0.$ This implies that the characteristic variety $p^{-1}(0)$ is a smooth conic submanifold of $T^*M\setminus 0$. Let $H_p$ be the Hamiltonian vector field 
\[ H_p:=\langle \partial_{\xi} p, \partial_x \rangle - \langle \partial_x p, \partial_{\xi} \rangle. \] 
Then $|\xi|H_p$ is a smooth vector field on $p^{-1}(0)$ that commutes with $\beta$ and is homogeneous of order $0$. We now define 
\[ \mathcal S:=\kappa(p^{-1}(0)), \ \mathrm{X}:= \kappa_*(|\xi|H_{p}). \]
$\mathrm{X}$ is a smooth vector field on $\mathcal S$. 
We assume that 
\begin{equation}
    \label{ap: msflow}
    \varphi_t:=e^{t\mathrm{X}} \text{ is a Morse--Smale flow with no fixed points on } \mathcal S.
\end{equation}
Recall that \eqref{ap: msflow} means that (see for instance \cite[Definition 5.1.1]{2dflow})
\begin{enumerate}[(1)]
    \item $\varphi_t$ has finitely many closed hyperbolic limit cycles;
    \item every trajectory of $\varphi_t$ that is not a limit cycle has unique limit cycles as its $\alpha, \omega$-limit sets.
\end{enumerate}
We denote the repulsive and attractive limit cycles by $\gamma_-$, $\gamma_+$ respectively and $\Lambda_{\pm}:=\kappa^{-1}(\gamma_{\pm})$. Then $\Lambda_{\pm}$ are conic Lagrangian submanifolds of $T^*M\setminus 0$. \cite[Lemma 2.1]{force} showed that $\gamma_{\pm}$ are the radial source $(-)$ and the radial sink $(+)$ respectively, for the Hamiltonian flow of $\abs{\xi}p$, in the sense of \cite[Definition E.50]{dz19}.

\Remark
\noindent
Since regular values and the Morse--Smale property are stable under small perturbations to the spectral parameter, if \eqref{ap: rv}, \eqref{ap: msflow} are satisfied, then there exists $\delta>0$ such that for every $\omega \in [-\delta, \delta]$, $\omega$ is a regular value of $p$ and \eqref{ap: msflow} is satisfied by 
\[ \mathcal S(\omega):=\kappa(p^{-1}(\omega)), \ \mathrm{X}(\omega):=\kappa_*(|\xi|H_p). \]
Moreover, let $\gamma_{\pm}(\omega)$ be the limit cycles, $\Lambda_{\pm}(\omega):=\kappa^{-1}(\gamma_{\pm}(\omega))$, then $\gamma_{\pm}(\omega)$ are the radial source ($-$) and the radial sink ($+$) for the Hamiltonian flow of $|\xi|(p-\omega)$. Moreover, if $K\subset T^*M\setminus 0$ satisfies (CC$\pm$) for $\mathbf P$, then it also satisfies (CC$\pm$) for $\mathbf P-\omega$, $\omega\in [-\delta,\delta]$. For $\omega\in \CC$, $\Re\omega\in [-\delta,\delta]$, $\Lambda_{\pm}(\omega)$ is understood as $\Lambda_{\pm}(\Re\omega)$.

We also assume 
\begin{equation}
    \label{ap: direction}
    q\leq 0 \text{ on } T^*M\setminus 0.
\end{equation}
In view of propagation of singularities \cite[Theorem E.47]{dz19}, \eqref{ap: direction} implies that the singularities propagate backwards along the Hamiltonian flow.
Let $\alpha\in C^{\infty}(T^*M;[0,1])$ such that $\supp \alpha\subset \{|\xi|\geq 1\}$, $\supp(1-\alpha)\subset \{|\xi|\leq 2\}$. Let $q_{-1}$ be the principal symbol of $Q-\mathrm{Op}(\alpha q)\in \Psi^{-1}(M)$:
\begin{equation}
    q_{-1}:=\sigma(Q-\Op(\alpha q)).
\end{equation}
Then $q_{-1}$ is homogeneous of order $-1$.

We now introduce the \emph{control conditions}:
\begin{defi}
    \label{def: gcc}
    For a conic set $K\subset T^*M\setminus 0$, we say that $K$ satisfies the forward (backward) control condition (CC$\pm$, $+$ for forward, $-$ for backward), if there exists a conic open subset $\Omega\subset T^*M\setminus 0$, such that $\Omega\subset K$ and every connected component of $\Lambda_+$ (or $\Lambda_-$ respectively) intersects $\Omega$. 
\end{defi}
\begin{figure}[t]
    \centering
    \subfigure[CC$+$]{\label{f1a}
    \includegraphics[scale=1, page=1]{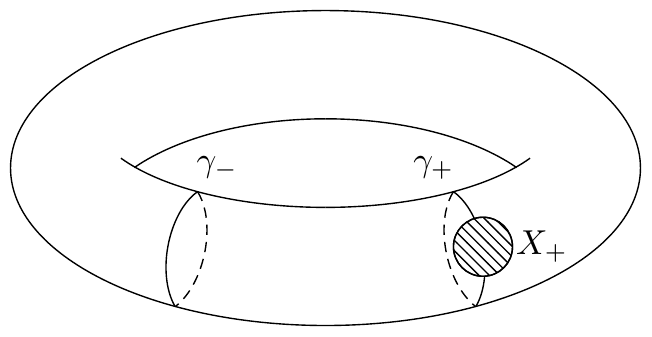}}
    \subfigure[CC$-$]{\label{f1b}
    \includegraphics[scale=1, page=2]{f1.pdf}}
    \caption{Illustration of the control conditions. $\gamma_+$ is the attractive limit cycle and $\gamma_-$ is the repulsive limit cycle, see \S \ref{subsec: assumptions}. The labels $X_{\pm}$ represent the radial compactification of the elliptic sets of $X_{\pm}$.}
    \label{fig: control}
\end{figure}

See Figures \ref{f1a} and \ref{f1b} for illustrations of the control conditions. 

\subsection{Main results}
\label{subsec: main}
Assume there is a conic subset of $T^*M\setminus 0$ that satisfies the control condition (CC$\pm$). In the corresponding Sobolev spaces, we can show the solution $u$ to the equation 
\begin{equation}
    (\mathbf P -\omega) u  = f\in C^{\infty}(M)
\end{equation}
can be microlocally controlled by its own part on this conic set, uniformly for $\omega$ close to $0$. More precisely:
\begin{theo}
    \label{thm: control}
    Suppose $\mathbf P\in \Psi^0(M)$ satisfies the conditions \eqref{ap: rv}, \eqref{ap: msflow}, \eqref{ap: direction} 
    and $q_{-1}$, $H_p$, $\Lambda_{\pm}(\omega)$ are as in \S \ref{subsec: assumptions}. Let (CC$\pm$) be as in Definition \ref{def: gcc}. Then the following are true:
    \begin{enumerate}[1.]
        \item Suppose $X_+\in \Psi^0(M)$ and $\Ell(X_+)$ satisfies (CC+). Then there exists $\delta>0$ such that for any 
            \begin{equation}
                \label{eq: hi-reg}
                |\Re \omega|\leq \delta, \ \Im\omega\geq 0, \ s>\sup_{\Lambda_-(\omega)}\left( -\tfrac{|\xi|q_{-1}}{H_p|\xi|} \right)-\tfrac12, \ N\in \RR,
            \end{equation}
            there exists $C>0$ independent of $\omega$ such that for any $u\in C^{\infty}(M)$,
            \begin{equation}
                \label{eq: sink-control}
                \|u\|_{H^s}\leq C \|X_+u\|_{H^s} + C\|(\mathbf P-\omega) u\|_{H^{s+1}} + C \|u\|_{H^{-N}}.
            \end{equation}

        \item Suppose $X_-\in \Psi^0(M)$ and $\Ell(X_-)$ satisfies (CC-). Then there exists $\delta>0$ such that for any
            \begin{equation}
                \label{eq: lo-reg}
                |\Re\omega|\leq \delta, \ \Im\omega\geq 0, \ s<\inf_{\Lambda_+(\omega)}\left( -\tfrac{|\xi|q_{-1}}{H_p|\xi|} \right)-\tfrac12, \ N\in \RR,
            \end{equation}
            there exists $C>0$ independent of $\omega$ such that for any $u\in C^{\infty}(M)$, 
            \begin{equation}
                \label{eq: source-control}
                \|u\|_{H^s}\leq C \|X_-u\|_{H^s}+C\|(\mathbf P-\omega) u\|_{H^{s+1}} + C\|u\|_{H^{-N}}.
            \end{equation}
    \end{enumerate}
\end{theo}

\Remarks
1. The regularity thresholds in \eqref{eq: hi-reg}, \eqref{eq: lo-reg} arise naturally from the radial estimates, see \cite[Theorem E.52, Theorem E.54]{dz19} for instance. When $\mathbf P$ takes the specific form \eqref{eq: dampp}, we have $q_{-1}=0$ and the thresholds in \eqref{eq: hi-reg}, \eqref{eq: lo-reg} are simply $s>-\frac12$, $s<-\frac12$ respectively.

\noindent 
2. Constants in the proofs of Theorem \ref{thm: control} -- including constants in the proofs of Lemma \ref{lem: source_cover} and Lemma \ref{lem: sink_cover} -- are all independent of $\omega$.

As an application of Theorem \ref{thm: control}, we study a microlocal model for the damped internal wave and show the elimination of the singular profile. More precisely, we put
\begin{equation} 
    \label{eq: dampp}
    \mathbf P=P-i\chi, \ \chi\in C^{\infty}(M;\RR_{\geq 0}),
\end{equation}
and we are interested in the long time behavior
of the solution to the equation 
\begin{equation}
    \label{eq: damp}
    (i\partial_t -\mathbf P)u(t,x)=f(x), \ u(0,x)=0.
\end{equation}
If $\chi\equiv 0$ and $0\notin \mathrm{Spec}_{\mathrm{pp}}(P)$, then 
by the main theorem of \cite{force} 
\begin{equation}\begin{gathered}\label{singular}
    u(t)=u_{\infty}+b(t)+\epsilon(t), \\
    u_{\infty}\in I^0(M;\Lambda_+(0)), \ \|b(t)\|_{L^2}\leq C, \ \|\epsilon(t)\|_{H^{-1/2-}}\to 0.
\end{gathered}\end{equation}
Thus the solution has a singular profile $u_{\infty}$, concentrating on the attracting Lagrangian submanifolds. Figure \ref{fig: undamped} exhibits the concentration.
\begin{figure}[t]
    \centering
    \subfigure[Undamped wave]{
        \label{fig: undamped}
    \includegraphics[scale=0.19]{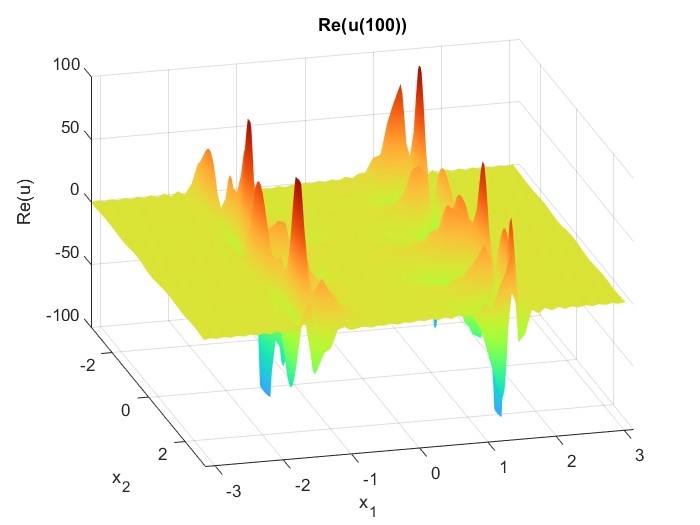}}
    \subfigure[Partially damped wave]{
        \label{fig: part-damped}
        \includegraphics[scale=0.19]{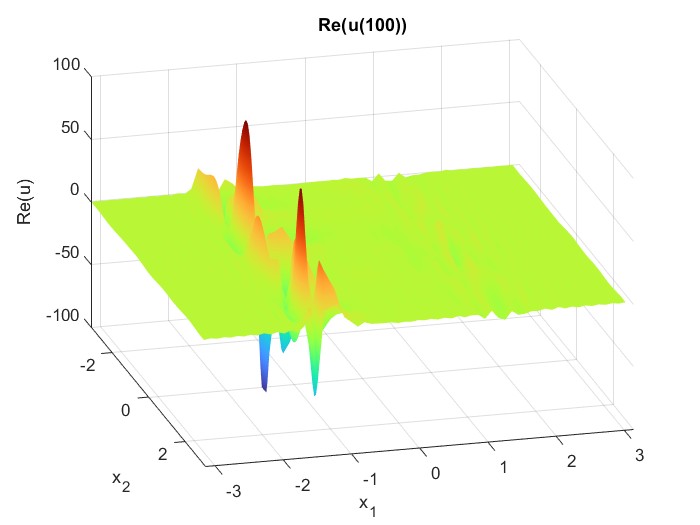}
    }
    \subfigure[Damped wave]{
        \label{fig: damped}
    \includegraphics[scale=0.19]{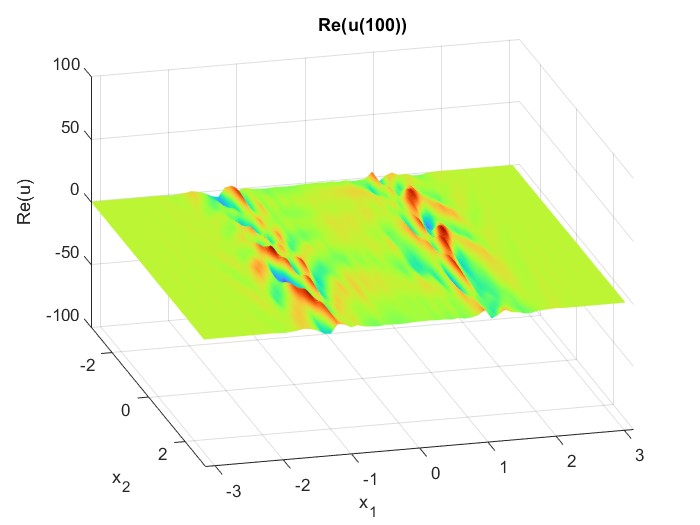}}
    \caption{Numerical illustration of solutions to \eqref{eq: damp} with $P$, $f$ given in \S \ref{subsec: examples} and $\chi = \chi_{j}$, $j=0,1,2$. (a) $\chi_0=0$. The undamped wave concentrates at the limit cycles and the $L^2$ norm of the solution blows up. (b) $\chi_1(x)=\tfrac12 e^{-5(x_1-\pi/2)^2}$. $\chi_1$ is numerically localized near $x_1=\frac{\pi}{2}$ and negligible at $x_1=-\frac{\pi}{2}$. As a result, the singular profile at $x_1=\frac{\pi}{2}$ disappears but there is still a singular profile at $x_1=-\frac{\pi}{2}$.  (c) $\chi_2(x_1,x_2)=\chi_1(x_1,x_2)+\chi_1(x_1+\pi, x_2)$. The damping is localized on both limit cycles and the solution stays bounded in $L^2$.}
    \label{fig: damp}
\end{figure}

The following theorem shows that if $0\notin \mathrm{Spec}_{\mathrm{pp}}(\mathbf P)$ and $\chi$ is nonzero on the projection of every connected component of the attracting Lagrangian to the base manifold, then the solution to \eqref{eq: damp} does not have singular profile: see Figure \ref{fig: damped}.

\begin{theo}
    \label{thm: damp}
    Let $\mathbf P$ be as in \eqref{eq: dampp} and satisfy \eqref{ap: rv}, \eqref{ap: msflow}, 
    and $0\notin \mathrm{Spec}_{\mathrm{pp}}(\mathbf P)$.
    Suppose $T^*(\supp \chi)\setminus 0$ satisfies (CC+) and $u$ solves \eqref{eq: damp}. Then there exists $C>0$ that does not depend on $t$ such that 
    \begin{equation}
        \|u(t)\|_{L^2(M)}\leq C, \ t>0.
    \end{equation}
\end{theo}

\Remarks 
\noindent
1. The assumption $0\notin \mathrm{Spec}_{\mathrm{pp}}(\mathbf P)$ is equivalent (see the proof of Lemma \ref{lem: finite-eig}) to the following form of the \emph{unique continuation principle} for $P=\Re(\mathbf P)$:
\begin{equation}
    \label{eq: unq-cont}
    Pu=0, \ u|_{\supp\chi}=0 \ \Rightarrow \ u\equiv 0.
\end{equation} 
In the case where $M=\mathbb T^2:=\RR^2/(2\pi \mathbb Z^2)$ and the full symbol of $P$ admits bounded analytic continuation $p(z,\zeta)$ from $T^*\mathbb T^2$ to 
\[ \left\{ (z,\zeta)\in \CC^2/(2\pi \mathbb Z^2)| \ |\Im z| \leq a_1, \ |\Im \zeta|\leq a_2 \langle \Re \zeta \rangle \right\}, \]
\cite[Proposition 3.1]{fermi} proved the analyticity of the eigenfunctions $u$ of $P$. Hence \eqref{eq: unq-cont} is satisfied in this case.

\noindent
2. For a general $\mathbf P$, \eqref{eq: unq-cont} can fail. In fact, let $\mathbf P= P-i\chi\in \Psi^0(M)$, $\chi\in C^{\infty}(M)$ such that $\mathbf P$ satisfies conditions in \S \ref{subsec: assumptions} and $M\setminus \supp \chi\neq \emptyset$. Then for any $v\in C^{\infty}(M;\RR)$, $\supp\chi\cap \supp v=\emptyset$, let $\Pi_{Pv}$ be the orthogonal projection on to $\CC Pv$. Put 
\begin{equation}
    \widetilde{\mathbf P}:=(P-\Pi_{Pv}P)- i\chi,
\end{equation}
Since $\Pi_{Pv}P \in \Psi^{-\infty}(M)$, we know $\widetilde{\mathbf P}$ satisfies conditions in \S \ref{subsec: assumptions} yet $\widetilde{\mathbf P}v=0$.

\noindent
3. One can also consider the case where some of the attracting Lagrangians are damped while the others are not. In this case, if $\Gamma_+$ is a connected component of $\Lambda_+(0)$ and $T^*(\supp\chi)\setminus 0$ contains a conic open subset that intersects $\Gamma$, then $u$ is bounded near $\Gamma_+$. If $\Gamma_+\cap T^*(\supp\chi)\setminus 0=\emptyset$ instead, then $u$ can have singular profiles near $\Gamma_+$ -- see Figure \ref{fig: part-damped}.

\subsection{Examples}
\label{subsec: examples}

\noindent
1. Consider $M=\mathbb T^2$, 
\begin{equation}\begin{gathered}
    \mathbf P=P-i\chi\in \Psi^0(\mathbb T^2), \ P:=\langle D \rangle^{-1}D_{x_2}-a\cos{x_1},\\ 
    a>0, \ a\neq 1, \ \chi\in C^{\infty}(\mathbb T^2), \ \chi\geq 0.
\end{gathered}\end{equation}
$P$ has the principal symbol $p(x,\xi)=|\xi|^{-1}\xi_2-a\cos{x_1}$.
The radially compactified characteristic variety $\mathcal S$ is a disjoint union of two tori. Let $\Pi_x: T^*\mathbb T^2\rightarrow \mathbb T^2$ be the projection of $(x,\xi)$ to $x$. If $0<a<1$, then $\mathcal S$ covers $\mathbb T^2$, meaning that $\Pi_x: \mathcal S\to \mathbb T^2$ is onto; if $a>1$, then $\mathcal S$ does not cover $\mathbb T^2$.
On $p^{-1}(0)\subset T^*\mathbb T^2\setminus 0$, we have
\[ H_{|\xi|p}=|\xi|H_p=-\frac{\xi_1\xi_2}{|\xi|^2}\partial_{x_1}+\frac{\xi_1^2}{|\xi|^2}\partial_{x_2}-a(\sin{x_1})|\xi|\partial_{\xi_1}. \]
We now introduce the coordinates $(\rho,\theta)\in \RR_+\times \mathbb S^1$ on $T^*\mathbb{T}^2\setminus 0$ such that $\xi_1=\rho\cos{\theta}$, $\xi_2=\rho\sin{\theta}$ and identify $\partial \overline{T^*\mathbb T^2}$ with the cosphere bundle $S^*\mathbb T^2$. Then 
\[ \mathcal S = \{ (x_1,x_2,1,\theta) | \ \sin{\theta} = a\cos{x_1}, \ x\in \mathbb T^2, \ \theta\in \mathbb S^1 \}, \ \mathrm{X} = -\sin{\theta}\cos{\theta}\partial_{x_1}+\cos^2{\theta}\partial_{x_2}+a\sin{x_1}\sin{\theta}\partial_{\theta}. \]
Now let 
\[\begin{gathered} 
    \gamma_-:=\left\{ (\tfrac{\pi}{2},x_2,1,0)| \ x_2\in \mathbb S^1 \right\}\cup \left\{ (-\tfrac{\pi}{2},x_2,1,\pi)| \ x_2\in \mathbb S^1 \right\}\subset \partial\overline{T^*\mathbb T^2}, \\
    \gamma_+:=\left\{ (\tfrac{\pi}{2}, x_2, 1,\pi)| \ x_2\in \mathbb S^1 \right\} \cup \left\{ (-\tfrac{\pi}{2},x_2,1,0)| \ x_2\in \mathbb S^1 \right\} \subset \partial\overline{T^*\mathbb T^2}.
\end{gathered}\]
Then we have $\mathrm{X}|_{\gamma_{\pm}} = \partial_{x_2}$. One can see now that $\gamma_{\pm}$ are closed orbits of $e^{t\mathrm{X}}$. On $\mathcal S$, near $\gamma_{\pm}$, the coefficient of $\partial_{\theta}$ is $\pm \sqrt{a^2-\sin^2{\theta}}\sin{\theta}$ where the signs $\pm$ are determined by the value of $x_1$. From here we see that $\gamma_-$ are repulsive cycles and $\gamma_+$ are attractive cycles. Now by definitions of $\Lambda_{\pm}(0)$, we know $H_{|\xi|p}$ has the attractive and the repulsive Lagrangian submanifolds
\[\begin{gathered} 
    \Lambda_-(0)=\kappa^{-1}(\gamma_-)=\left\{ (\pm \tfrac{\pi}{2}, x_2, \xi_1,0) | \ x_2\in \mathbb S^1, \ \pm \xi_1>0 \right\},\\ 
    \Lambda_+(0)=\kappa^{-1}(\gamma_+)=\left\{ (\pm \tfrac{\pi}{2}, x_2, \xi_1,0) | \ x_2\in \mathbb S^1, \ \pm \xi_1<0 \right\}. 
\end{gathered}\]
In Figure \ref{fig: damp}, we studied the numerical solution to \eqref{eq: damp} with 
\[ a=\tfrac12, \ f(x)= -5(e^{ -3((x_1+0.9)^2+(x_2+0.8)^2)}+e^{ -3((x_1-0.9)^2+(x_2-0.8)^2)})e^{2ix_1+ix_2} \]
and 
\[ \chi_0\equiv 0, \ \chi_1(x)=\tfrac12 e^{-5(x_1-\pi/2)^2}, \chi_2(x)=\tfrac12 e^{-5(x_1-\pi/2)^2}+\tfrac12 e^{-5(x_1+\pi/2)^2} \]
for Figure \ref{fig: undamped}--\ref{fig: damped} respectively. We use \texttt{MATLAB} to produce the numerical solutions and we refer to \cite{numerics} for the numerical schemes used here.

\noindent
2. In \eqref{eq: damp}, if $\mathbf P=P-i$, $f\in C^{\infty}(M)$, then from the spectral theory we know
\[ u(t)=(P-i)^{-1}(e^{-t-itP}-I)f(x). \]
Since $e^{-itP}$ preserves the $L^2$ norm and $P$ is of 0th order, we know $u\in C^{\infty}([0,\infty);L^2(M))$ and 
\[ \|u(t)\|_{L^2}\leq (1+e^{-t})\|f\|_{L^2}. \]
Moreover we have 
\[ \lim_{t\to \infty}u(t) = -\mathbf P^{-1}f(x) \text{ in } L^2(M). \]
If, however, the damping term $-i$ is removed in this example and $\mathbf P = P$, then one can only claim that $u(t)\in H^{-\frac12-}(M)$ by \eqref{singular}.

\noindent
{\bf Acknowledgement.}
The authors are thankful to Jared Wunsch and Maciej Zworski for their insightful comments on the failure of unique continuation for general zeroth order operators. The authors are grateful to two anonymous referees for kindly reading this manuscript and providing many valuable remarks. RPTW is partially supported by NSF grant DMS-2054424.

\section{Microlocal control estimates}
In this section, we prove Theorem \ref{thm: control}. We always assume $\omega=\lambda+i\epsilon$ with $\lambda\in [-\delta,\delta]$ where $\delta$ satisfies the conditions in Remark in \S \ref{subsec: assumptions} and $\epsilon>0$. We remark that results in this section are uniform in $\omega$ by the stability of the Morse--Smale flow discussed in Remark in \S \ref{subsec: assumptions}.

Let $\psi_t:= e^{tH_{|\xi|(p-\lambda)}}: T^*M\setminus 0\to T^*M\setminus 0$ be the Hamiltonian flow of $|\xi|(p-\lambda)$ on $T^*M\setminus 0$. 

To state the following lemmata for covering $T^*M\setminus 0$ by the images of elliptic sets of operators under propagation, we recall that for a pseudodifferential operator $A\in \Psi^k(M)$, its elliptic set $\Ell(A)$ is defined as a conic subset of $T^*M\setminus 0$, such that $(x_0,\xi_0)\in \Ell(A)$ if and only if there exists a conic neighborhood $U$ of $(x_0,\xi_0)$ in $T^*M\setminus 0$, such that 
\[ |\sigma(A)(x,\xi)|\geq C|\xi|^k, \ (x,\xi)\in U, \]
where $\sigma(A)$ is the principal symbol of $A$. 

\begin{figure}[t]
    \centering
    \subfigure[Covers in Lemma \ref{lem: sink_cover}]{
    \includegraphics[scale=1,page=3]{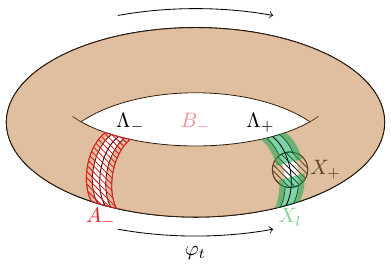}}
    \subfigure[Covers in Lemma \ref{lem: source_cover}]{
    \includegraphics[scale=1,page=4]{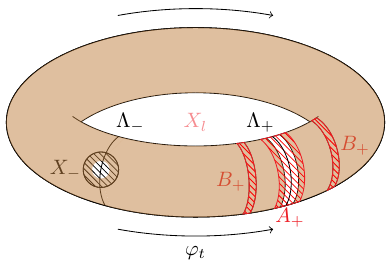}}
    \caption{Illustration of the covers constructed in Lemma \ref{lem: sink_cover} and Lemma \ref{lem: source_cover}. The labels should be understood as the radial compactifications of the elliptic sets of the corresponding operators.}
    \label{fig: cover}
\end{figure}

We then have
\begin{lemm}
    \label{lem: sink_cover}
    Suppose $A_-$, $X_+\in \Psi^0(M)$ such that
    \[ \Lambda_-\subset \Ell(A_-), \ \Ell(X_+) \text{ satisfies (CC$+$). } \]
    Then there exist finitely many operators $\{B_{\ell}, X_{m}\}_{\ell,m}\subset \Psi^0(M)$, such that the elliptic sets of $A_-, B_{\ell}, X_+, X_{m}$ cover  $p^{-1}(\lambda)\subset T^*M\setminus 0$
    and for each $\ell$, $m$, there exist $0<T_{\ell}, T^{\prime}_{m}<\infty$, such that
    \begin{equation}
        \label{eq: prin-prop}
        \WF(B_{\ell})\subset \psi_{T_{\ell}}(\Ell(A_-)), \ \WF(X_{m})\subset \psi_{T^{\prime}_{m}}(\Ell(X_+)).
    \end{equation}
\end{lemm}

\begin{proof}
    Since the bicharacteristics of $|\xi|(p-\lambda)$ are the same as the ones for $p-\lambda$ in $T^*M\setminus 0$, it suffices to consider the radial compactification (by $\kappa$ in \S \ref{subsec: assumptions}) of all the conic sets in the lemma, and replace $\psi_t$ by $\varphi_t$. We now identify the wavefront sets and elliptic sets for operators with their radial compactification. 

    Let $\Omega_+$ be an open subset of $\Ell(X_+)$ such that $\overline{\Omega_+}\subset \Ell(X_+)$ and every connected component of $\gamma_+$ intersects $\Omega_+$. From the Morse--Smale assumption \eqref{ap: msflow}, we know each connected component of $\gamma_+$ is a closed orbit under the flow $\varphi_t$, and intersects $\Omega_+$. This implies that there is $T>0$ such that $\cup_{t=0}^T \varphi_{t}(\Omega_+)$ covers the closed orbit, and its compactness gives a finite subcover. Taking the union of the finite subcovers over the finitely many connected component of $\gamma_+$, we have finitely many $T^{\prime}_{m}$, such that 
    \[ \gamma_+\subset U_+:=\bigcup_{m}\varphi_{T^{\prime}_{m}}(\Omega_+). \]
    Now let $U_-\subset T^*M\setminus 0$ be an open neighborhood of $\gamma_-$ such that $\overline{U_-}\subset \Ell(A_-)$. 

    Now for any $z\in \mathcal S\setminus U_+$, by \eqref{ap: msflow}, there exists $T(z)$ such that $z\in\varphi_{T(z)}(U_-)$. Notice that $\{\varphi_{T(z)}(U_-)\}_{z\in \mathcal S\setminus U_+}$ is an open cover of the compact set $\mathcal S\setminus U_+$, we can extract a finite cover $\{\varphi_{T(z_{\ell})}(U_-)\}$ of $\mathcal S\setminus U_+$.

    Let $b_{\ell}$, $\rho_{m}\in C^{\infty}(T^*M\setminus 0;\RR)$ be functions that are homogeneous of order $0$ and 
    \[ \supp b_{\ell}\subset \varphi_{T(z_{\ell})}(\Ell(A_-)), \ b_-|_{\overline{\varphi_{T(z_{\ell})}(U_-)}}=1; \ \supp \rho_{m}\subset \varphi_{T^{\prime}_{m}}(\Ell(X_+)), \ \rho_{m}|_{\overline{\varphi_{T^{\prime}_{m}}(\Omega_+)}}=1.  \]
    Put
    \[ B_{\ell}:=\Op(b_{\ell}), \ X_{m}:=\Op(\rho_{m}) \]
    and we complete the proof.
\end{proof}

\begin{lemm}
    \label{lem: source_cover}
    Suppose $A_+$, $B_+$, $X_-\in \Psi^0(M)$ such that 
    \[ \Lambda_+\subset \Ell(A_+), \ \Lambda_+\cap \WF(B_+)=\emptyset, \ \Ell(X_-) \text{ satisfies (CC$-$). } \]
    Then there exist finitely many operators $\{X_{m}\}_m\subset \Psi^0(M)$, such that the elliptic sets of $ A_+, X_-, X_{m}$, cover $ p^{-1}(\lambda)\subset T^*M\setminus 0 $, and the elliptic sets of $ X_-, X_{m}$, cover $\WF(B_+)$. Moreover, for each $m$ there exist $0<T_{m}<\infty$, such that 
    \begin{equation}
        \WF(X_{m}) \subset \psi_{T_{m}}(\Ell(X_-)).
    \end{equation}
\end{lemm}
The proof for Lemma \ref{lem: source_cover} is similar to the proof of Lemma \ref{lem: sink_cover}.

We now prove the control estimates in Theorem \ref{thm: control}.
\begin{proof}[Proof of Theorem \ref{thm: control}]
    For $\omega=\lambda+i\epsilon$, $\lambda\in [-\delta,\delta]$, $\epsilon>0$, we denote 
    \[\begin{gathered} 
        \mathbf P_1(\omega):= \langle D \rangle^{1/2}(\mathbf P -\omega)\langle D \rangle^{1/2}, \ u_1:=\langle D \rangle^{-1/2}u\in C^{\infty}, \\ 
        P_1:=\langle D \rangle^{1/2}(P-\lambda)\langle D \rangle^{1/2}, \ Q_1:=\langle D \rangle^{1/2}Q \langle D \rangle^{1/2}. 
    \end{gathered}\]

    \textbf{Principal type of propagation estimates.}
    Suppose $A, B\in \Psi^0(M)$ such that there exists $T>0$ such that 
    \[ \psi_{-T}(\WF(B))\subset \Ell(A). \]
    Then for any $s$, $N\in \RR$, there exists $C>0$ such that for any $u_1\in C^{\infty}(M)$, we have 
    \begin{equation}
        \label{eq: prop}
        \|B u_1\|_{H^s}\leq C \|A u_1\|_{H^s} + C \|\mathbf P_1(\omega) u_1\|_{H^s} + C\|u_1\|_{H^{-N}}.
    \end{equation}
    Here \eqref{eq: prop} follows from \cite[Theorem E.47]{dz19} and the fact that 
    \begin{equation}
        \Im \sigma(\mathbf P_1(\omega))=|\xi|(q-\epsilon)\leq 0.
    \end{equation}

    \textbf{Radial source estimates.}
    The goal of this step is to show the following estimate: there exists $A_-\in \Psi^0(M)$ such that $\Lambda_-\subset \Ell(A_-)$ and for any $r,N\in \RR$ such that 
    \begin{equation}
        r>\sup_{\Lambda_-(\lambda)}\left( -\frac{|\xi|q_{-1}}{H_p|\xi|} \right),
    \end{equation}
    there exists $C>0$ such that for any $u_1\in C^{\infty}(M)$, we have 
    \begin{equation}
        \label{eq: source-est}
        \|A_-u_1\|_{H^r}\leq C\|\mathbf P_1(\omega)u_1\|_{H^r}+C\|u_1\|_{H^{-N}}.
    \end{equation}

    The proof of \eqref{eq: source-est} is a modification of the proof of \cite[Theorem E.52]{dz19} and \cite[\S 3]{force}. Indeed, let $B_-\in \Psi^0(M)$ such that $\Lambda_-\subset \Ell(B_-)$, $\WF(B_-)\cap \Lambda_+=\emptyset$ and $\rho\in C^{\infty}(T^*M\setminus 0;[0,1])$ be an \emph{escape function} such that $\rho$ is homogeneous of order 0 and 
    \[  \supp\rho\subset \Ell(B_-), 
     \rho =1 \text{ near } \Lambda_-, \ H_{|\xi|(p-\lambda)}(\rho)\leq 0. \]
     For the construction of such functions, see \cite[Lemma C.1]{zeta} or \cite[Lemma E.53]{dz19}.

    Now for $r\in \RR$, we put $G:=\Op(\langle \xi \rangle^r \rho)\in \Psi^r(M)$. Then 
    \[\begin{split} 
        \Im \langle \mathbf P_1(\omega)u_1, G^*G u_1 \rangle
        = & \langle -\tfrac{1}{2i}[P_1, G^*G]u_1,u_1 \rangle 
        + \langle (Q_1-\epsilon\langle D \rangle)Gu_1, Gu_1 \rangle \\
        & + \langle \Re\left( G^*[G,Q_1-\epsilon\langle D \rangle]\right)u_1, u_1 \rangle. 
    \end{split}\] 
    Notice that $\alpha\sqrt{-|\xi|q}\in S^{1/2}(T^*M;\RR)$, where $\alpha$ is given in \S \ref{subsec: assumptions}. Let $Q_{1/2}:=\Op(\alpha \sqrt{-|\xi| q})\in \Psi^{1/2}(M)$, then
    \[ Q_1+ Q_{1/2}^*Q_{1/2}\in \Psi^{0}(M), \ \sigma(Q_1+Q_{1/2}^*Q_{1/2})=|\xi|q_{-1}. \]
    Now we have
    \begin{equation}\begin{split}
        \label{eq: com}
        & \Im \langle \mathbf P_1(\omega) u_1, G^*G u_1 \rangle \\
        & =   \langle Tu_1, u_1 \rangle - \| Q_{1/2} G u_1 \|^2_{L^2} -\epsilon \|\langle D \rangle^{1/2}G u_1\|_{L^2}^2 + \langle \Re(G^*[G,Q_1-\epsilon\langle D \rangle])u_1,u_1 \rangle, \\
        & \leq  \langle Tu_1, u_1 \rangle + \langle \Re(G^*[G,Q_1-\epsilon\langle D \rangle])u_1,u_1 \rangle
    \end{split}\end{equation}
    with 
    \[T:=-\tfrac{1}{2i}[P_1, G^*G]+G^*(Q_1+Q_{1/2}^*Q_{1/2})G \in \Psi^{2r}(M), \ T^*=T.\]
    \begin{enumerate}[(i)]
        \item Since $G^*[G,Q_1-\epsilon\langle D \rangle]\in \Psi^{2r}(M)$ has purely imaginary principal symbol, we know 
        \[ \Re(G^*[G,Q_1-\epsilon\langle D \rangle])\in \Psi^{2r-1}(M), \ \WF(\Re(G^*[G,Q_1-\epsilon\langle D \rangle]))\subset \Ell(B_-). \]
        By the elliptic estimates \cite[Theorem E.33]{dz19}, we have 
        \begin{equation}
            \label{eq: ell}
            \langle \Re(G^*[G,Q_1-\epsilon\langle D \rangle])u_1, u_1 \rangle \leq C \|B_-u_1\|_{H^{r-1/2}}^2+C\|u_1\|^2_{H^{-N}}.
        \end{equation}

        \item The principal symbol of $T$ is 
            \begin{equation}
                \sigma(T) = |\xi|^{2r}\rho \left[ H_{|\xi|(p-\lambda)}\rho + \rho \left( rH_p|\xi|+|\xi|q_{-1} \right) \right].
            \end{equation}
            Let $r$ satisfies 
            \begin{equation}
                \label{eq: high-reg}
                r>\max_{\Lambda_-(\lambda)}\left( -\tfrac{|\xi|q_{-1}}{H_p|\xi|} \right),
            \end{equation}
            then there exists $\nu>0$ and a conic open neighborhood $U\subset B_-$ of $\Lambda_-(\lambda)$ such that 
            \begin{equation}
                rH_p|\xi|+|\xi|q_{-1}\leq -\nu \text{ on } U.
            \end{equation}
            Thus 
            \begin{equation}
                \sigma(T+\nu G^*G)|_U \leq 0. 
            \end{equation}
            Now apply the microlocal G\aa rding inequality \cite[Proposition E.34]{dz19}, with $(A,B,B_1)$ there replaced by $(-T-\nu G^* G, 0, B_-)$ and $h=1$, we have  
            \begin{equation}
                \label{eq: gar}
                \nu \|G u_1\|_{L^2}^2 \leq -\langle T u_1, u_1 \rangle +C\|B_- u_1\|_{H^{r-1/2}}^2 + C\|u_1\|_{H^{-N}}^2.
            \end{equation}
    \end{enumerate}

    Combining \eqref{eq: com}, \eqref{eq: ell} and \eqref{eq: gar} and we find 
    \begin{equation}
         \nu\|Gu_1\|_{L^2}^2\leq -\Im \langle G \mathbf P_1(\omega)u_1, G u_1 \rangle + C \|B_- u_1\|_{H^{r-1/2}}^2 + C\|u_1\|_{H^{-N}}^2.
    \end{equation}
    Use Cauchy--Schwartz and we have 
    \begin{equation}
        \|G u_1\|_{L^2}\leq C\|G\mathbf P_1(\omega) u_1\|_{L^2}+C\|B_- u_1\|_{H^{r-1/2}} + C\|u_1\|_{H^{-N}}.
    \end{equation}
    Let $A_-:=\Op(\rho)$, then by the elliptic estimates \cite[Theorem E.33]{dz19}
    \begin{equation}
        \label{eq: prop1}
        \|A_-u_1\|_{H^r}\leq C\| \mathbf P_1(\omega) u_1 \|_{H^r}+ C\|B_-u_1\|_{H^{r-1/2}}+C\|u_1\|_{H^{-N}}.
    \end{equation}

    To remove the $\|B_-u_1\|_{H^{r-1/2}}$ term, we use the propagation estimates \eqref{eq: prop}
    and find 
    \begin{equation}
        \|B_-u_1\|_{H^{r-1/2}}\leq C\|A_-u_1\|_{H^{r-1/2}}+C\|\mathbf P_1(\omega)u_1\|_{H^{r-1/2}}+C\|u_1\|_{H^{-N}}.
    \end{equation}
    Apply the interpolation inequality for $H^{r-1/2}$ and $H^{r}$, $H^{-N}$ to the term $\|A_-u_1\|_{H^{r-1/2}}$, and we have 
    \begin{equation}
        \|B_-u_1\|_{H^{r-1/2}}\leq \tfrac12 \|A_-u_1\|_{H^{r}}+C\|\mathbf P_1(\omega) u_1\|_{H^r}+C\|u_1\|_{H^{-N}}.
    \end{equation}
    This together with \eqref{eq: prop1} gives \eqref{eq: source-est}.

    \textbf{Radial sink estimates.}
    Similar to \eqref{eq: source-est}, we can also prove the following radial sink estimates for $\mathbf P_1(\omega)$: there exist $A_+, B_+\in \Psi^0(M)$ such that $\Lambda_+\subset \Ell(A_+)$, $\WF(B_+)\cap \Lambda_+=\emptyset$ and for any $N\in \RR$,
    \[ r<\inf_{\Lambda_+(\lambda)}\left( -\frac{|\xi| q_{-1}}{H_p|\xi|} \right) \]
    there exists $C>0$ such that for any $u_1\in C^{\infty}$, we have 
    \begin{equation}
        \label{eq: sink-est}
        \|A_+ u_1\|_{H^r} \leq C\|B_+u_1\|_{H^r} + C\|\mathbf P_1(\omega) u_1\|_{H^r} + C \|u_1\|_{H^{-N}}.
    \end{equation}

    \textbf{Control by the sink.}
    Suppose $X_+\in \Psi^0(M)$ and $\Ell(X_+)$ satisfies (CC$+$). We now prove \eqref{eq: sink-control}. 

    Let $\widetilde{X}_+\in \Psi^0(M)$ be such that $\WF(\widetilde{X}_+)\subset \Ell(X_+)$ and $\Ell(\widetilde{X}_+)$ satisfy (CC$+$). Let $A_-\in \Psi^0(M)$ be as in \eqref{eq: source-est}. Then $A_-$, $\widetilde{X}_+$ satisfy conditions in Lemma \ref{lem: sink_cover} and we can find $B_{\ell}$, $X_m\in \Psi^0(M)$ satisfying conditions in Lemma \ref{lem: sink_cover}. Then by \eqref{eq: prop}, for any $r, N\in \RR$ and each $\ell$, $m$, there exists $C>0$ such that for any $u_1\in C^{\infty}$, we have 
    \begin{equation}\begin{gathered}
        \label{eq: sink-prop}
        \|B_{\ell}u_1\|_{H^r}\leq C\|A_-u_1\|_{H^r} + C\|\mathbf P_1(\omega) u_1\|_{H^r} +C\|u_1\|_{H^{-N}}, \\
        \|X_m u_1\|_{H^r}\leq C\|\widetilde{X}_+ u_1\|_{H^r}+C\|\mathbf P_1(\omega) u_1\|_{H^r} + C\|u_1\|_{H^{-N}}.
    \end{gathered}\end{equation} 
    Combine \eqref{eq: source-est}, \eqref{eq: sink-prop}, the fact that the elliptic sets of $A_-$, $B_{\ell}$, $\widetilde{X}_+$, $X_m$ covers $p^{-1}(\lambda)$, and the elliptic estimates \cite[Theorem E.33]{dz19}, we know that for $r$ satisfying \eqref{eq: high-reg}
    \begin{equation}
        \|u_1\|_{H^r}\leq C\|\widetilde{X}_+ u_1\|_{H^r} + C\|\mathbf P_1(\omega) u_1\|_{H^r} + C \|u_1\|_{H^{-N}}.
    \end{equation}
    The control estimate \eqref{eq: sink-control} now follows from the elliptic estimate applied to $\widetilde{X}_+\langle D \rangle^{\frac12}$ and $\langle D \rangle^{\frac12}X_+$.

    \textbf{Control by the source.} The control estimate \eqref{eq: source-control} is proved in a similar way. 
    
    Assume $X_-\in \Psi^0(M)$ satisfies (CC$-$). Let $\widetilde{X}_-\Psi^0(M)$ be such that $\WF(\widetilde{X}_-)\subset \Ell(X_-)$ and $\Ell(\widetilde{X}_-)$ satisfy (CC$-$). Let $A_+$, $B_+\in \Psi^0(M)$ be as in \eqref{eq: sink-est}. We then apply Lemma \ref{lem: source_cover} to $A_+$, $B_+$, $\widetilde{X}_-$ and get operators $X_m$. 
    
    By \eqref{eq: prop} we find for $r$, $N\in \RR$,
    \begin{equation}
        \label{eq: source-prop}
        \|X_m u_1\|_{H^r}\leq C\|\widetilde{X}_- u_1\|_{H^r} + C\|\mathbf P_1(\omega) u_1\|_{H^r} + C \|u_1\|_{H^{-N}}.
    \end{equation}
    The estimate \eqref{eq: source-control} is now a result of \eqref{eq: sink-est}, \eqref{eq: source-prop} and elliptic estimates.
\end{proof}

\section{The limiting absorption principle}

From now on we assume
\begin{equation} 
    \label{ap: p-chi}
    \mathbf P=P - i\chi, \ \chi\in C^{\infty}(M;\RR_{\geq 0}). 
\end{equation}
We ask $T^*(\supp \chi)\setminus 0$ to satisfy the control condition (CC+). This is equivalent to say that 
\begin{equation} 
    \label{ap: chi-control}
    \supp\chi \text{ intersects with each connected component of } \pi(\Lambda_+).  
\end{equation}
Here $\pi: T^*M\setminus 0\to M$ is the natural projection. 
We remark that in this case, $q_{-1}=0$. Hence the regularity conditions \eqref{eq: hi-reg}, \eqref{eq: lo-reg} become
\[ s>-\tfrac12, \text{ or } s<-\tfrac12 \]
respectively.

\begin{lemm}
    \label{lem: finite-eig}
    Suppose $\mathbf P$ satisfies \eqref{ap: rv}, \eqref{ap: msflow}, \eqref{ap: p-chi}, \eqref{ap: chi-control}. Then there exists $\delta>0$ such that 
    \[ |\mathrm{Spec}_{L^2, \mathrm{pp}}(\mathbf P) \cap [-\delta,\delta]|<\infty. \]
\end{lemm}
\begin{proof}
    By \cite[Lemma 3.2]{force}, it suffices to show that 
    \[ \{ \Re \omega \ | \ \omega \in \mathrm{Spec}_{L^2, \mathrm{pp}}(\mathbf P), \Im\omega\geq 0 \} \subset \mathrm{Spec}_{L^2, \mathrm{pp}}(P). \]
    In fact, assume $\omega\in \CC$, $\Im \omega\geq 0$, $u\in \mathrm{Ker}_{L^2}(\mathbf P-\omega)$.
    Then 
    \[ 0=\Im \langle (\mathbf P-\omega)u, u \rangle = -\langle (\chi+\Im \omega) u, u \rangle= - \int_{M} (\chi+\Im \omega) |u|^2 dx. \]
    Since $\chi+\Im \omega\geq 0$, we know $(\chi+\Im \omega) u=0$. Thus $(P-\Re \omega)u=0$.
\end{proof}

We record a lemma on the regularity of solutions.
\begin{lemm}
    \label{lem: unique}
    Let $\mathbf P$ be as in Lemma \ref{lem: finite-eig}. Then there exists $\delta>0$ such that
    \[ (\mathbf P-\omega)u\in C^{\infty}, \ u\in \mathscr D^{\prime}(M), \ \omega\in [-\delta,\delta] \]
    implies that $u\in C^{\infty}(M)$.
\end{lemm}
\begin{proof}
    Let $\widetilde{\chi}\in C^{\infty}(M)$ satisfy $\supp\widetilde{\chi}\Subset \{\chi>0\}$ and \eqref{ap: chi-control}. Then notice that 
    \[\sigma(\mathbf P-\omega)=p-\omega-i\chi,\]
    which is elliptic on $\supp\widetilde{\chi}$. Hence by the elliptic estimate, for $s, N\in \RR$, there exists $C>0$, such that 
    \begin{equation}
        \|\widetilde{\chi}u\|_{H^s}\leq C\|(\mathbf P-\omega)u\|_{H^s}+ C\|u\|_{H^{-N}}.
    \end{equation}
    This implies $\widetilde{\chi}u\in C^{\infty}(M)$. The desired result $u\in C^{\infty}(M)$ now follows from \eqref{eq: sink-control} with $X_+$ replaced by $\widetilde{\chi}$.
\end{proof}

Now we show the limiting absorption principle:
\begin{lemm}
    \label{lem: lap}
    Suppose $\mathbf P$ is as in Lemma \ref{lem: finite-eig} and $0\notin \mathrm{Spec}_{\mathrm{pp}}(\mathbf P)$.
    Then there exists $\delta>0$ such that for any $\omega\in [-\delta,\delta]$, $f\in C^{\infty}(M)$, the limit 
    \[ (\mathbf P-\omega-i0)^{-1}f:=\lim_{\epsilon\to 0+}(\mathbf P-\omega-i\epsilon)^{-1}f \]
    exists in $C^{\infty}(M)$. Moreover, 
    \begin{equation}
        (\mathbf P-\omega-i0)^{-1}f\in C^{\infty}_{\omega}((-\delta, \delta);C^{\infty}(M))
    \end{equation}
    and is the unique solution to the equation 
    \begin{equation}
        (\mathbf P-\omega)u=f, \ u\in \mathscr D^{\prime}(M).
    \end{equation}
\end{lemm}

\begin{proof}[Proof of Lemma \ref{lem: lap}]
    \noindent
    1. \emph{The limit exists and is smooth on $M$}. By Lemma \ref{lem: finite-eig}, for $\delta>0$ sufficiently small, we have $\mathrm{Spec}_{\mathrm{pp}}(\mathbf P)\cap [-\delta,\delta]=\emptyset$. We first show that for any $\omega\in [-\delta,\delta]$, $f\in C^{\infty}(M)$, we have 
    \[ (\mathbf P-\omega-i0)^{-1}f\in C^{\infty}(M). \]
    For simplicity,  we put $\omega=0$.

    For any $\epsilon>0$, since $\chi\geq 0$, we know $\mathbf P-i\epsilon$ is elliptic. Thus
    $ u_{\epsilon}:=(\mathbf P-i\epsilon)^{-1}f\in C^{\infty}(M)$.
    Let $\widetilde{\chi}$ be as in the proof of Lemma \ref{lem: unique}.
    Theorem \ref{thm: control} implies that for any $s>-\frac12$, $N\in \RR$, there exists $C>0$ such that for any $\epsilon>0$, 
    \begin{equation}
        \label{eq: d-control}
        \|u_{\epsilon}\|_{H^s}\leq C\| \widetilde{\chi} u_{\epsilon} \|_{H^s}+C\|f\|_{H^{s+1}}+C\|u_{\epsilon}\|_{H^{-N}}.
    \end{equation}
    On the other hand, for $\epsilon>0$, we have 
    \[ \Im\sigma(\mathbf P-i\epsilon)=-(\chi+\epsilon), \]
    which is uniformly elliptic over $\supp\widetilde{\chi}$. Thus by the elliptic estimates \cite[Theorem E.33]{dz19}, for any $s$, $N\in \RR$, there exists $C>0$ such that
    \begin{equation}
        \label{eq: d-ell}
        \|\widetilde{\chi} u_{\epsilon}\|_{H^s}\leq C\|f\|_{H^s}+C\|u_{\epsilon}\|_{H^{-N}}.
    \end{equation}
    Estimates \eqref{eq: d-control}, \eqref{eq: d-ell} implies that for any $s>-\frac12$,
    \begin{equation}
        \label{eq: controlled-hi-reg}
        \|u_{\epsilon}\|_{H^s}\leq C\|f\|_{H^{s+1}}+C\|u_{\epsilon}\|_{H^{-N}}.
    \end{equation}
    
    We now show that for any $s>-\frac12$, $\{\|u_{\epsilon}\|_{H^s}\}_{\epsilon>0}$ is bounded. Otherwise, there exists $s>-\frac12$ and a subsequence $u_{\ell}$ such that $\|u_{\ell}\|_{H^s}\to \infty$, $\ell\to \infty$. We put $\widetilde{u}_{\ell}:=u_{\ell}/\|u_{\ell}\|_{H^s}$, then 
    \begin{equation} 
        \label{eq: rescale}
        \|\widetilde{u}_{\ell}\|_{H^s}=1, \ (\mathbf P-i\epsilon)\widetilde{u}_{\epsilon} = f/\|u_{\ell}\|_{H^s} \xrightarrow{C^{\infty}} 0, \ \ell\to\infty. 
    \end{equation}
    By \eqref{eq: controlled-hi-reg}, for any $s^{\prime}>-\frac12$, we have 
    \begin{equation} 
        \label{eq: tilde-smooth}
        \|\widetilde{u}_{\ell}\|_{H^{s^{\prime}}} \leq C \|f\|_{H^{s^{\prime}+1}}/\|u_{\ell}\|_{H^s}+C\|\widetilde{u}_{\ell}\|_{H^s} = C\|f\|_{H^{s^{\prime}+1}}/\|u_{\ell}\|_{H^s} + C. 
    \end{equation}
    In particular, we know $\widetilde{u}_{\ell}$ is bounded in $H^{s^{\prime}}$ for $s^{\prime}>s$. Since the embedding $H^{s^{\prime}}\hookrightarrow H^s$ is compact when $s^{\prime}>s$, by passing to a subsequence, we can assume that $u_{\ell}\to u$ in $H^s$. Let $\ell\to \infty$ in \eqref{eq: rescale}, we find 
    \[ \|u\|_{H^s}=1, \ \mathbf P u=0. \]
    This contradicts Lemma \ref{lem: unique} and the assumption $0\notin \mathrm{Spec}_{\mathrm{pp}}(\mathbf P)$.

    We conclude now that $\{u_{\epsilon}\}_{\epsilon>0}$ is bounded in $H^s$ for any $s>-\frac12$ and hence in $H^s$ for any $s\in \RR$. A similar argument as above shows that $\{u_{\epsilon}\}_{\epsilon>0}$ is precompact in $H^s$ for any $s\in \RR$. Notice that by $(\mathbf P-i\epsilon)u_{\epsilon}=f$ and \eqref{eq: controlled-hi-reg}, every limit point $u$ has to satisfy 
    \begin{equation} 
        \label{eq: limit-sol}
        \mathbf P u=f, \ u\in \mathscr D^{\prime}(M).
    \end{equation}
    By either \eqref{eq: controlled-hi-reg} or Lemma \ref{lem: unique}, we can see that $u\in C^{\infty}$. 
    By Lemma \ref{lem: unique}, we know such $u$ is unique. Hence $u_{\epsilon}$ converges to the unique solution to \eqref{eq: limit-sol}.

    2. \emph{Smoothness in $\omega$}. First, to see that $(\mathbf P-\omega-i0)^{-1}f$ is continuous in $\omega$ for $\omega\in (-\delta,\delta)$, we can replace $u_{\epsilon}=(\mathbf P-i\epsilon)^{-1}f$ above by $u_{\ell}:=(\mathbf P-\omega_{\ell}-i\epsilon_{\ell})^{-1}f$ with 
    \[ \epsilon_{\ell}>0, \ \omega_{\ell}\in (-\delta,\delta), \ \omega_{\ell}+i\epsilon_{\ell}\to \omega\in (-\delta,\delta). \]
    The previous argument shows that $\{u_{\ell}\}$ converges in $H^s$ for any $s$ to the unique solution to 
    \[ u\in C^{\infty}, \ (\mathbf P-\omega)u=f, \]
    which is $(\mathbf P-\omega-i0)^{-1}f$. This implies $(\mathbf P-\omega-i0)^{-1}f$ is continuous in $\omega\in (-\delta,\delta)$.

    For any $k\in \mathbb N$, we denote 
    \[ (\mathbf P-\omega-i0)^{-k}f:=\left((\mathbf P-\omega-i0)^{-1}\right)^k f. \]
    Then we have $(\mathbf P-\omega-i0)^{-k}f\in C^{\infty}$. We claim that for $\omega\in (-\delta,\delta)$,
    \begin{equation}
        \label{eq: omega-deri}
        \partial_{\omega}^k\left( (\mathbf P-\omega-i0)^{-1}f \right)=k!(\mathbf P-\omega-i0)^{-k-1}f, \ k\geq 1.
    \end{equation}
    In fact, for $k=1$, $\omega_0\in (-\delta,\delta)$ we have 
    \[ \frac{(\mathbf P-\omega-i0)^{-1}f-(\mathbf P-\omega_0-i0)^{-1}f}{\omega-\omega_0}=(\mathbf P-\omega-i0)^{-1}(\mathbf P-\omega_0-i0)^{-1}f. \]
    Let $\omega\to \omega_0$ and use the continuity in $\omega$, and we find 
    \[ \partial_{\omega}|_{\omega=\omega_0}(\mathbf P-\omega-i0)^{-1}f=(\mathbf P-\omega_0-i0)^{-2}f. \]
    \eqref{eq: omega-deri} then follows by induction in $k$. This concludes the proof.
\end{proof}

\section{The damped equation}

In this section we study the evolution problem for $\mathbf P=P-i\chi$, which considered as the damping problem for $P$ and prove Theorem \ref{thm: damp}.

\begin{proof}[Proof of Theorem \ref{thm: damp}]
    Since $\mathbf P: L^2(M)\to L^2(M)$ is bounded, we define 
    \begin{equation}
        U(t):=\sum_{\ell=0}^{\infty}\frac{(-i t \mathbf P)^{\ell}}{\ell!}: L^2(M) \to L^2(M), \ t\in \RR.
    \end{equation}
    Then $U(t)$ is also bounded for any $t\in \RR$.

    We first show that \eqref{eq: damp} has a unique solution 
    \begin{equation}
        \label{eq: damp-solution}
        u(t):=(U(t)-I)(\mathbf P-i0)^{-1}f.
    \end{equation}
    Indeed, by Lemma \ref{lem: lap}, we have $(\mathbf P-i0)^{-1}f\in C^{\infty}\subset L^2(M)$. We thus have $u(t)\in L^2(M)$ and can check that
    \[ (i\partial_t - \mathbf P)u(t) = f, \ u(0)=0. \]
    Suppose \eqref{eq: damp} has another solution $w(t)$, then we have 
    \[ (i\partial_t-\mathbf P)(u-w)=0, \ u(0)-w(0)=0. \]
    Now we compute 
    \[\begin{split} 
        0= & 2\Im\langle (i\partial_t-\mathbf P)(u-w), u-w \rangle_{L^2(M)} \\ 
        & = \partial_t \|u-w\|^2_{L^2(M)} +2\langle \chi(u-w), u-w \rangle_{L^2(M)} 
         \geq \partial_t \|u-w\|_{L^2(M)}^2.
    \end{split}\]
    Thus we know 
    \[ \|u(t)-w(t)\|_{L^2(M)}\leq \|u(0)-w(0)\|_{L^2(M)}=0 \ \Rightarrow \ u=w, \ t\geq 0. \]

    To see that $\|u(t)\|_{L^2(M)}$ is uniformly bounded in $t$, we notice that 
    \[ (i\partial_t-\mathbf P)(u(t)+(\mathbf P-i0)^{-1}f)=0. \]
    We again compute 
    \[ 0=\Im\langle (i\partial_t-\mathbf P)(u+(\mathbf P-i0)^{-1}f), u+(\mathbf P-i0)^{-1}f \rangle\geq \partial_t\|u+(\mathbf P-i0)^{-1}f\|_{L^2(M)}^2. \]
    Thus 
    \[ \|u+(\mathbf P-i0)^{-1}f\|_{L^2(M)}\leq \|(\mathbf P-i0)^{-1}f\|_{L^2(M)}, \]
    which implies 
    \[ \|u(t)\|_{L^2(M)}\leq 2\|(\mathbf P-i0)^{-1}f\|_{L^2(M)}. \]
    This concludes the proof of Theorem \ref{thm: damp}.
\end{proof}

\bibliographystyle{alpha}
\bibliography{Robib,Jianbib,HC-bib}

\end{document}